\documentclass[a4paper,9pt]{article}
\usepackage[plainpages=false]{hyperref}
\usepackage{amsfonts,latexsym,rawfonts,amsmath,amssymb,amsthm,mathrsfs}
\usepackage{amsmath,amssymb,amsfonts,latexsym,lscape,rawfonts,natbib}

\usepackage[all]{xy}
\usepackage{eufrak}
\usepackage{makeidx}         
\usepackage{graphicx,psfrag}

\newcommand{\D}[2]{\ensuremath{ \frac{\partial{#1}}{\partial{#2}}}}
\newcommand{\R}{\ensuremath{\mathbb{R}}}

\DeclareMathOperator{\supp}{supp}
\DeclareMathOperator{\Area}{Area}
\DeclareMathOperator{\Vol}{Vol}
\DeclareMathOperator{\diam}{diam}

\DeclareMathOperator{\Lap}{\triangle}

\DeclareMathOperator{\diag}{diag}
\DeclareMathOperator{\Exp}{exp}

\title{On the Conditions to Extend Ricci Flow}
\author{Bing Wang}
\date{}
\begin{document}

\maketitle

\newtheorem{definition}{Definition}[section]
\newtheorem{result}{Result}[section]
\newtheorem{lemma}{Lemma}[section]
\newtheorem{theorem}{Theorem}[section]
\newtheorem{corollary}{Corollary}[section]
\newtheorem{remark}{Remark}[section]
\newtheorem{proposition}{Proposition}[section]
\newtheorem{property}{Property}[section]
\newtheorem{claim}{Claim}[section]
\newtheorem{conjecture}{Conjecture}[section]
\newtheorem{question}{Question}[section]
\newtheorem{example}{Example}[section]

\newcommand{\bo}[1]{\ensuremath{\bar{\omega}(#1)}}
\newcommand{\norm}[2]{{ \ensuremath{\|} #1 \ensuremath{\|}}_{#2}}

\newcommand{\lafa}{\ensuremath{\lambda_\alpha}}
\newcommand{\xa}{\ensuremath{x_\alpha}}
\newcommand{\ya}{\ensuremath{y_\alpha}}
\newcommand{\ga}{\ensuremath{g_\alpha}}
\newcommand{\Ra}{\ensuremath{R_\alpha}}
\newcommand{\fa}{\ensuremath{f_\alpha}}

\begin{abstract}
   Consider $\{ (M^n,g(t)), \; 0 \leq t < T<\infty\}$ as an
   unnormalized Ricci flow solution: $\frac{dg_{ij}}{dt} = -2R_{ij}$
   for $t \in [0,T)$. Richard Hamilton shows that if the curvature
   operator is uniformly bounded under the flow for all $t \in
   [0,T)$ then the solution can be extended over $T$. Natasa Sesum
   proves that a uniform bound of Ricci tensor is enough to extend
   the flow.  We show that if Ricci is bounded from below, then a
   scalar curvature integral bound is enough to extend flow, and
   this integral bound condition is optimal in some sense.
\end{abstract}

\section{ When can Ricci flow be extended?}
In~\cite{Ha1}, R. Hamilton introduces  Ricci flow which deforms
Riemannian metrics in the direction of the Ricci tensor. One hopes
that the Ricci flow will deform any Riemannian metric to some
canonical metrics, such as Einstein metrics. One can even understand
geometric and topological structure of the underlying differential
manifold by this sort of deformation. The idea is best illustrated
in~\cite{Ha1} where Hamilton proves that in any simply connected 3
manifold without boundary, any Riemannian metric with positive Ricci
curvature can be deformed into a positive space form (up to
scaling). Consequently, R. Hamilton proves that the underlying
manifold is indeed diffeomorphic to $S^3.\;$  This fundamental work
sparks a great interest of many mathematicians in Ricci flow. In a
series of work, R. Hamilton introduces an ambitious program to prove
the Poincar\`e conjecture via Ricci flow (cf.~\cite{Ha3} for
Hamilton's program and early references in Ricci flow.). The
celebrated work of G. Perelman~\cite{Pe1},~\cite{Pe2} and~\cite{Pe3}
indeed proves the Poincar\`e conjecture which states that every
simply connected 3 manifold is $S^3.\;$ We refer the readers
to~\cite{KL},~\cite{MT}
for more information. \\

After Perelman's work in the Ricci flow,  there is a renewed
interest in Ricci flow and its application around the world.  We
will refer readers to the book~\cite{CLN} for more updated
references.  In this note, we want to concentrate in studying some
basic issue on Ricci flow: the maximal existence time of Ricci flow
and  the geometric conditions that might affect the maximal
existence
time.\\

One notes that Ricci flow is a weak Parabolic flow.  R. Hamilton
first proves that for any smooth initial data, the flow will exist
for a short time in~\cite{Ha1}. In~\cite{De},  Hamilton's proof is
simplified greatly by a clever choice of gauge. The next immediate
question is the so called ``maximal existence time" for the Ricci
flow (with respect to initial metric). In~\cite{Ha3}, Hamilton
proves that if $T<\infty$ is the maximal existence time of a closed
Ricci flow solution $\{ (M^n,g(t)), 0 \leq t <T < \infty \}$, then
Riemannian curvature is unbounded as $t \to T$. In other words, a
uniform bound for Riemannian curvature on $M \times [0,T)$ is enough
to extend Ricci flow over time $T$. In~\cite{Se}, by a blowup
argument, Sesum shows that Ricci curvature uniformly bounded  on $M
\times [0,T)$ is enough to extend Ricci flow over $T$. Sesum's
surprising work uses the no local collapsing theorem of Perelman. A
natural question arises: what is the optimal condition for the Ricci
flow to be extended?  In many ways, we believe that the scalar
curvature bound shall be enough to extend the flow.    In this note,
we first prove ( See Definition~\ref{definition: 1} for notations),

\begin{theorem}
$\{ (M^n,g(t)), 0 \leq t <T < \infty \}$ is a closed Ricci flow
solution.
 If
 \begin{enumerate}
 \item[1.]    $Ric(x,t) \geq -A$  for all $(x,t) \in M \times
 [0,T)$, $A$ is a positive constant ,
 \item[2.]    $\norm{R}{\alpha,\; M \times [0,T)} < \infty, \;
 \alpha \geq \frac{n+2}{2}$,
 \end{enumerate}
 then this flow can be extended over time $T$.
\label{theorem: 1}
\end{theorem}

and

\begin{theorem}
 $\{(M^n,g(t)), 0 \leq t <T < \infty \}$ is a closed Ricci flow solution.
If
\begin{align*}
 \norm{Rm}{\alpha,\; M \times [0,T)} < \infty, \quad \alpha \geq
 \frac{n+2}{2},
\end{align*}
\label{theorem: 2} then this flow can be extended over time $T$.
\end{theorem}

These two theorems are optimal in some aspects as illustrated by
Example~\ref{example: 1} in the next section.

\begin{remark}
  In theorem~\ref{theorem: 1},~\ref{theorem: 2}, let
$\alpha = \infty$, we can recover  Sesum's and Hamilton's results.
\end{remark}


\noindent {\bf Organization} Let's sketch the outline of this note.
We first fix some notations in section 2. Then, in section 3, we
prove Theorem~\ref{theorem: 2} for all $n \geq 2$. In section 4, we
use no local collapsing theorem and Croke's argument to establish a
local Sobolev constant control. Then we use this control to develop
a general parabolic
 Moser iteration under Ricci flow in section 5.  Applying Moser iteration
 to $R$  in section 6, we prove Theorem~\ref{theorem: 1} for $n \geq
 3$.\\

\noindent{\bf Acknowledgements}: I would like to express my
gratitude to  my advisor, professor Xiuxiong Chen. He directed me to
this subject and brought this problem to my attention and even
showed me the main tools to handle this problem. I'm grateful to
professor Dan Knopf and professor Sigurd Angenent for their helpful
discussions.
I would like to thank Haozhao Li for pointing out some errors in the
ealier version of this note.

\section{Preliminary}
Let $M^n$ be a connected compact manifold without boundary.
$(M^n,g(t))$ is called a closed Ricci flow solution if the metric
satisfies the equation:
\begin{align}
  \frac{d g_{ij}}{dt} = -2R_{ij}.
  \label{eqn: ricciflow}
\end{align}
 By direct calculation, we have the evolution equations for
curvatures under Ricci flow:
\begin{align}
  \D{R}{t} &= \Lap R + 2|Ric|^2, \label{eqn: evolutionr}\\
  \D{R_{ij}}{t} &= \Lap R_{ij} + 2 R_{iklj}R_{kl} - 2 R_{ik}
  R_{kj}, \label{eqn: evolutionricci}\\
  \D{R_{ijkl}}{t} &= \Lap R_{ijkl} +2(B_{ijkl}-B_{ijlk} +B_{ikjl} -B_{iljk})  \notag \\
   & \quad \quad       -(R_{ip}R_{pjkl}
   +R_{jp}R_{ipkl}+R_{kp}R_{ijpl}+R_{lp}R_{ijkp}),
  \label{eqn: evolutionriemannian}
\end{align}
where $B_{ijkl} \triangleq -R_{ipqj} R_{kpql}$.

The evolution equation of volume element is
\begin{align}
 \D{d\mu}{t} = -R d\mu. \label{eqn: evolutionv}
\end{align}

For convenience, we define some norm of the space time manifold $M
\times [0,T)$ below.

\begin{definition}
Suppose $N \subset M$, for any measurable function $F$ defined on $N
\times [0, T)$ and $\alpha \geq 1$, we define
\begin{align*}
 &\norm{F}{\alpha,N \times [0,T)} \triangleq  (\int_0^T \int_N |F|^\alpha d\mu dt)^{\frac{1}{\alpha}},\\
 &\norm{F(\cdot,t)}{\alpha,N}  \triangleq (\int_N |F|_{g(t)}^\alpha d\mu_{g(t)})^{\frac{1}{\alpha}}, \\
 &F_{+} \triangleq \max \{F,0 \},\qquad
 F_{-} \triangleq \max \{-F,0 \}.
\end{align*}
\label{definition: 1}
\end{definition}

Now we are ready to give example to illustrate that Theorem 1.1 is
sharp in some aspects.
\begin{example}
     Let $ (S^n,g_s)$ be the space form of constant sectional curvature $1$.
      Now we start Ricci flow from metric $ (S^n,g_s)$. By direct calculation,
       $g(t)= (1-2(n-1)t)g_s$ is the Ricci flow solution. Therefore,
        $T=\frac{1}{2(n-1)}$ is the maximal existence time. However, we compute

\begin{align*}
 \norm{R}{\alpha,\; M \times [0, T)} &= \{\int_0^T \int_M |R|^\alpha d\mu dt \}^{\frac1\alpha}\\
 &=\{\int_0^T V(t) (\frac{n}{2(T-t)})^\alpha dt \}^{\frac1\alpha}\\
 &=\frac{n}{2} {V(0)}^{\frac{1}{\alpha}} T^{-\frac{n}{2 \alpha}} \{  \int_0^T (T-t)^{\frac{n}{2}-\alpha}  dt\}^{\frac1\alpha},
\end{align*}
therefore,
\begin{align*}
   \norm{R}{\alpha,\; M \times [0, T)}
   \left\{
   \begin{array}{ll}
    &=\infty, \quad  \alpha \geq \frac{n}{2}+1,\\
    &<\infty, \quad  \alpha < \frac{n}{2}+1.
   \end{array}
   \right.
\end{align*}
Moreover, $Ric \geq 0$. This suggests us that Theorem~\ref{theorem:
1} cannot be improved to $\alpha <\frac{n+2}{2}$.

Since $S^n$ is space form, $|Rm|^2 ={C(n)}^2 |R|^2$, then
\begin{align*}
 \norm{Rm}{\alpha,\; M \times [0, T)}= C(n) \norm{R}{\alpha,\; M \times [0, T)}.
\end{align*}
 Hence,
\begin{align*}
  \norm{Rm}{\alpha,\; M \times [0, T)}
   \left\{
   \begin{array}{ll}
    &=\infty, \quad  \alpha \geq \frac{n}{2}+1,\\
    &<\infty, \quad  \alpha < \frac{n}{2}+1.
   \end{array}
   \right.
\end{align*}
This implies Theorem~\ref{theorem: 2} can not be improved to $\alpha
<\frac{n+2}{2}$. \label{example: 1}
\end{example}

The uniform Sobolev constant control will play an important role in
our proof.

\begin{definition}
  Suppose $\{(M^n,g(t)), 0 \leq t <T < \infty \}$ is a closed Ricci flow
  solution, $ N \subsetneq M$. We say $\sigma$ is a uniform Sobolev constant for $N$ at each
  time slice, if
\begin{align}
   (\int_N |v|_{g(t)}^{\frac{2n}{n-2}} d\mu_{g(t)})^{\frac{n-2}{n}} \leq
       \sigma \int_N |\nabla v|_{g(t)}^2  d\mu_{g(t)},
   \label{eqn: sobolev}
\end{align}
for every function $ v \in W_0^{1,2}(N)$ and $0 \leq t < T$.
\end{definition}

If Ricci is bounded from below, we can control $\D{R}{t}$ by $R$.

\begin{property}
Suppose $Ric \geq -B$, let $\hat{R} = R+nB$, then
\begin{align}
   \D{\hat{R}}{t}  \leq  \Lap \hat{R} +2(\hat{R} -2B )\hat{R} +2nB^2.
\label{eqn: rhat}
\end{align}
\end{property}
\begin{proof}
 Choose an orthonormal basis to diagonalize $Ric$  such that
  $Ric = \diag\{\lambda_1, \cdots , \lambda_n\}$,
 then
\begin{align*}
  Ric + B I = \diag \{ \lambda_1+B, \cdots, \lambda_n+B \},
\end{align*}
where each term is nonnegative. Therefore,
\begin{align*}
    (\lambda_1+B)^2+ \cdots (\lambda_n+B)^2 \leq (\lambda_1+B +\cdots +\lambda_n
    +B)^2,
\end{align*}
consequently,
\begin{align*}
  \lambda_1^2 +\cdots + \lambda_n^2
  \leq  (\lambda_1+ \cdots +\lambda_n)^2 + 2(n-1)B(\lambda_1+ \cdots +
  \lambda_n) +n(n-1)B^2,
\end{align*}
i.e.
\begin{align}
   |Ric|^2 &\leq R^2 +2(n-1)BR +n(n-1)B^2 \notag \\
      &=  \hat{R}^2 -2B \hat{R} + nB^2.
\label{eqn: riccirhat}
\end{align}

From inequality (\ref{eqn: evolutionr}), we have
\begin{align*}
   \D{\hat{R}}{t}  &= \D{R}{t}  \\
           &= \Lap R + 2 |Ric|^2  \\
           &\leq \Lap \hat{R} +2(\hat{R}^2 -2B \hat{R} +nB^2) \\
           &= \Lap \hat{R} +2(\hat{R} -2B )\hat{R} +2nB^2.
\end{align*}
\end{proof}

  In~\cite{Pe1}, Perelman proves the fundamental
  no local collapsing Theorem:
\begin{theorem}
  $\{(M^n,g(t)), 0 \leq t <T < \infty \}$ is a closed Ricci flow
  solution. Then there exists a $\kappa >0$, such that for any
  $(x, t) \in M \times [0,T), r>0$, if
$ \sup_{y \in B_{g(t)}(x, r)}|Rm|(y,t) \leq r^{-2},$ then
\begin{align*}
   \frac{\Vol_{g(t)}(B_{g(t)}(x,r)) }{r^n} \geq \kappa.
\end{align*}
\label{theorem: noncollapsing1}
\end{theorem}

Actually, Perelman has already noticed that the same conclusion
still holds if we replace the Riemannian curvature by scalar
curvature. That is the next theorem.
\begin{theorem}
  $\{(M^n,g(t)), 0 \leq t <T < \infty \}$ is a closed Ricci flow
  solution. Then there exists a $\kappa >0$, such that for any
  $(x, t) \in M \times [0,T), r>0$, if
$ \sup_{y \in B_{g(t)}(x, r)}|R(y,t)| \leq r^{-2},$ then
\begin{align*}
   \frac{\Vol_{g(t)}(B_{g(t)}(x,r)) }{r^n} \geq \kappa.
\end{align*}
\label{theorem: noncollapsing2}
\end{theorem}
The proof of Theorem~\ref{theorem: noncollapsing2} can be found
in~\cite{KL}, ~\cite{ST}. We will use Theorem~\ref{theorem:
noncollapsing2} to get Sobolev constant control.

\section{ Proof of Theorem~\ref{theorem: 2} for $n \geq 2$}
\begin{proof}
By H\"{o}lder's inequality, $\norm{Rm}{\alpha, M \times [0,T)} <
\infty $ implies $\norm{Rm}{\frac{n+2}{2}, M \times [0,T)} < \infty$
if $\alpha > \frac{n+2}{2}$. So we only need to prove
Theorem~\ref{theorem: 2} for $\alpha = \frac{n+2}{2}$.

 We argue by contradiction.

 Suppose $T$ is the maximal existence
time. Then there is a sequence  $(x^{(i)}, t^{(i)})$ with $\lim_{i
\to \infty} t^{(i)} =T$ and $\lim_{i \to \infty}|Rm|^{(i)} =\infty$.
Moreover,
\begin{align*}
    |Rm|(x^{(i)}, t^{(i)}) = \max_{(x,t) \in M \times [0,
    t^{(i)}]} |Rm|(x,t).
\end{align*}
Let
\begin{align*}
Q^{(i)} \triangleq |Rm|(x^{(i)}, t^{(i)}),\\
g^{(i)}(t) \triangleq Q^{(i)} g((Q^{(i)})^{-1} t + t^{(i)}).
\end{align*}
  By Theorem~\ref{theorem: noncollapsing1}, we have uniform lower bound  of injectivity radius at points $(x^{(i)},t^{(i)})$ for the sequence $\{((M^n, x^{(i)}), g^{(i)}(t)), \; -Q^{(i)}t^{(i)} \leq t \leq 0 \}$ . So it
  subconverges to an ancient Ricci flow solution $\{ ((\bar{M}, \bar{x}), \bar{g}(t)),-\infty \leq t \leq 0
  \}$.
Therefore, by the scaling invariance of $\int_0^T \int_M
|Rm|^{\frac{n+2}{2}} d\mu dt$, we have
\begin{align}
 \int_{-1}^0 \int_{B_{\bar{g}(0)(\bar{x}, 1)}} |\bar{Rm}|^{\frac{n+2}{2}} d \bar{\mu} dt
 &\leq \lim_{i \to \infty} \int_{-1}^0 \int_{B_{g^{(i)}(0)}(x^{(i)}, 1)} |Rm|_{g^{(i)}(t)}^{\frac{n+2}{2}} d
 \mu_{g^{(i)}(t)}
 dt \notag \\
 &=\lim_{i \to \infty} \int_{t^{(i)}-(Q^{(i)})^{-1}}^{t^{(i)}} \int_{B_{g(t^{(i)})}(x^{(i)}, (Q^{(i)})^{-\frac12})}
 |Rm|^{\frac{n+2}{2}} d\mu dt \notag \\
 &\leq \lim_{i \to \infty} \int_{t^{(i)}-(Q^{(i)})^{-1}}^{t^{(i)}} \int_M |Rm|^{\frac{n+2}{2}} d\mu
 dt  \notag \\
 &=0. \label{eqn: Rmvanish}
\end{align}
The last equality holds since $\int_0^T \int_M |Rm|^{\frac{n+2}{2}}
d\mu dt < \infty$ and $\lim_{i \to \infty} (Q^{(i)})^{-\frac12} =0$.
Since $(\bar{M}, \bar{g}(t))$ is a smooth Riemannian manifold for
each $t \leq 0$, equality (\ref{eqn: Rmvanish}) implies that
$|\bar{Rm}| \equiv 0$ on the parabolic ball $B_{\bar{g}(0)(\bar{x},
1)} \times [-1,0]$. In particular, $|\bar{Rm}|(\bar{x},0) =0$. On
the other hand,
\begin{align*}
  |\bar{Rm}|(\bar{x},0) = \lim_{i \to \infty} |Rm|_{g^{(i)}}(x^{(i)},
  0) =1.
\end{align*}
So we get a contradiction.
\end{proof}

When dimension is $2$, $Rm=R$. Thus  Theorem~\ref{theorem: 1} and
Theorem~\ref{theorem: 2} are the same. So we have already proved
Theorem~\ref{theorem: 1} for $n=2$.   When $n \geq 3$, $R$ and $Rm$
are different. Accordingly we have to develop some new techniques to
prove Theorem~\ref{theorem: 1}. Moser iteration will play a critical
role in our proof. In order to apply Moser iteration, we need to get
a local Sobolev constant control first.

\section{Local Sobolev Constant Control}

 In this section, we discuss how to control isoperimetric constant
 locally. By the equivalence of  isoperimetric constant and
 Sobolev constant, we get the local control for Sobolev constant. The
 following argument comes from Croke's paper~\cite{Cr}.

 \begin{definition}
   Suppose $(N,\partial N, g)$ be a smooth compact manifold with smooth
   boundary and Riemannian metric $g$.
   \begin{align*}
      \Phi(N) \triangleq \inf_{\Omega \Subset N} \frac{
      \Area(\partial{\Omega})^{n}}{\Vol(\Omega)^{n-1}}.
   \end{align*}
 Let $UN \stackrel{\pi}{\to} N$ represent the unit sphere bundle with the canonical
 measure. For $v \in UN$, let $\gamma_v$ be the geodesic with $\gamma_v'(0)
 =v$, let $\zeta^t(v)$ represent the geodesic flow, i.e. $\zeta^t(v) =
 \gamma_v'(t)$.  Let $l(v)$ be the smallest value of $t >0$ (possibly $\infty$) such
 that $\gamma_v(t) \in \partial N$.  Note $\zeta^t(v)$ is defined
 for $t \leq l(v)$. Let
 \begin{align*}
 \tilde{l}(v) & \triangleq \sup \{t | \gamma_v \; \textrm{minimizes up to} \; t \;
\textrm{and} \;t \leq l(v) \}, \\
 \tilde{U}M & \triangleq \{ v \in UM | \tilde{l}(v) = l(v)\},
 \qquad \tilde{U}_p  \triangleq \pi|_{\tilde{U}M}^{-1}(p),\\
 \tilde{\omega}_p & \triangleq   \frac{\Area{\tilde{U}_p}}{\Area{U_p}}, \qquad
 \tilde{\omega}  \triangleq \inf_{p \in N} \tilde{\omega}_p,\\
 \alpha(n) & \triangleq \textrm{ volume of unit sphere of dimension
 } n.
 \end{align*}

 For $p \in \partial N$, let $N_p$ be the inwardly pointing unit
 normal vector. Let $U^+ \partial N \to \partial N$ be the bundle of
 inwardly pointing unit vectors. That is,
\[
  U^+ \partial N = \{ u \in UN|_{\partial N} | \langle u, N_{\pi(u)} \rangle \geq 0
  \}.
\]
  $U^+ \partial N$ has natural metric structure.
 \end{definition}

  This $\tilde{\omega}$ is related to $\Phi(N)$ closely. If we have
  a control over $\tilde{\omega}$, then it's easy to get a control
  for $\Phi(N)$.

\begin{proposition}
    For $(N,\partial N, g)$ we have
\begin{align}
 \int_{\tilde{U}N} f(v) dv = \int_{U^+ \partial N}
 \int_0^{\tilde{l}(u)} f(\zeta^r(u)) <u, N_{\pi(u)}> dr du,
 \label{eqn: integral1}
\end{align}
where $f$ is any integrable function. In particular for $f \equiv
1$, we have
\begin{align}
 \Vol(\tilde{U}M) = \int_{U^+ \partial N} \tilde{l}(u) \langle u,
 N_{\pi(u)} \rangle du.
 \label{eqn: integral2}
\end{align}
\end{proposition}
This formula occurs in~\cite{Be1}, p.286, and ~\cite{Sa},
pp.336-338.

\begin{proposition}
   Let $N^n$ be a Riemannian manifold and $u \in UN$. Then for every
   $l \leq C(u)$ (the distance to the cut locus in the direction
   $u$):
\begin{align}
   \int_{x=0}^{x= l} \int_{z=0}^{z=l-x} F(\zeta^x(u), z) dz dx \geq
   C_1(n)\cdot \frac{l^{n+1}}{\pi^{n+1}},
   \label{eqn: integral3}
\end{align}
where $C_1(n) = \frac{\pi \alpha(n)}{2 \alpha(n-1)}$. $F(v,z)$ is
the volume form in polar coordinates,i.e.,
\[
    \int_{U_p} \int_0^{C(v)} F(v,z) dz dv = \Vol(M).
\]
\end{proposition}
 The proof can be found in Berger's work~\cite{Be2} (Appendix D).

\begin{lemma}
 For $(N,\partial N, g)$ we have the isoperimetric inequality:
\begin{align}
   \frac{\Area(\partial N)^n}{\Vol(N)^{n-1}} \geq C_2(n)
   \tilde{\omega}^{n+1},
   \label{eqn: psi}
\end{align}
where $C_2(n) = 2^{n-1} \frac{\alpha(n-1)^n}{\alpha(n)^{n-1}}$.
\label{lemma: psicontrol}
\end{lemma}
\begin{proof}
\begin{align}
 \Vol(N)^2 &= \int_N \Vol(N) dp  \notag\\
 &\geq \int_N \int_{U_p}\int_0^{\tilde{l}(u)} F(u,t) dt du  dp \notag\\
 &= \int_{UN} \int_0^{\tilde{l}(u)} F(u,t) dt du  \notag\\
 &\geq \int_{\tilde{U}N} \int_0^{\tilde{l}(u)} F(u,t) dt du \notag\\
 &= \int_{U^+ \partial N} \int_0^{\tilde{l}(v)}
 \int_0^{\tilde{l}(\zeta^s(v))} F(\zeta^s(v),t) \langle v,
 N_{\pi(v)} \rangle dt ds dv  \notag\\
 &= \int_{U^+ \partial N} [\int_0^{\tilde{l}(v)} \int_0^{\tilde{l}(v) -s} F(\zeta^s(v), t) dtds]
         \langle v, N_{\pi(v)} \rangle dv \notag\\
 &\geq \frac{C_1(n)}{\pi^{n+1}} \int_{U^+ \partial N}
 (\tilde{l}(v))^ {n+1} \langle v, N_{\pi(v)} \rangle dv.
 \label{eqn: volume2}
\end{align}
By H\"{o}lder inequality,
\begin{align*}
  \int_{U^+\partial N} \tilde{l}(v) \langle v, N_{\pi(v)} \rangle dv
  & = \int_{U^+\partial N} (\tilde{l}(v)\langle v, N_{\pi(v)} \rangle
  ^{\frac{1}{n+1}})  \langle v, N_{\pi(v)} \rangle ^{\frac{n}{n+1}}
  dv\\
  &\leq (\int_{U^+\partial N} \tilde{l}^{n+1} \langle v, N_{\pi(v)} \rangle
  dv)^{\frac{1}{n+1}} (\int_{U^+\partial N} \langle v, N_{\pi(v)} \rangle dv)^{\frac{n}{n+1}},
\end{align*}
then,
\begin{align}
\int_{U^+\partial N} \tilde{l}^{n+1} \langle v, N_{\pi(v)} \rangle
dv \geq \frac{(\int_{U^+\partial N} \tilde{l}(v) \langle v,
N_{\pi(v)} \rangle dv)^{n+1}} {(\int_{U^+\partial N} \langle v,
N_{\pi(v)} \rangle dv)^n}. \label{eqn: holder}
\end{align}
Put inequality (\ref{eqn: holder}) into inequality (\ref{eqn:
volume2}), we get
\begin{align*}
   \Vol(N)^2 &\geq \frac{C_1(n)}{\pi^{n+1}}
   \frac{(\int_{U^+\partial N} \tilde{l}(v) \langle v,
N_{\pi(v)} \rangle dv)^{n+1}} {(\int_{U^+\partial N} \langle v,
N_{\pi(v)} \rangle dv)^n},\\
\end{align*}
therefore,
\begin{align*}
\Vol(N)^2 (\int_{U^+\partial N} \langle v, N_{\pi(v)} \rangle dv)^n
  &\geq \frac{C_1(n)}{\pi^{n+1}} \Vol(\tilde{U}M)^{n+1}\\
  &\geq \frac{C_1(n)}{\pi^{n+1}}[\tilde{\omega} \alpha(n-1) \Vol(N)]^{n+1}.
\end{align*}
Note that
\begin{align*}
  \int_{U^+\partial N} \langle v, N_{\pi(v)} \rangle dv
= \frac{\alpha(n)}{2 \pi} \Area(\partial N),
\end{align*}
consequently,
\begin{align*}
 \frac{\Area(\partial N)^n}{\Vol(N)^{n-1}}
   &\geq \frac{C_1(n)}{\pi^{n+1}} \tilde{\omega}^{n+1} \alpha(n-1)^{n+1}
   \frac{(2\pi)^n}{\alpha(n)^n}\\
   &= 2^{n-1} \frac{\alpha(n-1)^n}{\alpha(n)^{n-1}} \tilde{\omega}^{n+1} \\
   &\triangleq C_2(n) \tilde{\omega}^{n+1}.
\end{align*}
\end{proof}

\begin{lemma}
 $M$ is a complete Riemannian manifold with $Ric \geq -(n-1)K^2$. $\Omega \Subset N_1 \subset N_2 \subset M$, $\Omega$ is a domain with smooth boundary, and $\diam(N_2) \leq D$.  Then
\begin{align}
   \tilde{\omega}(\Omega) \geq \frac{\Vol(N_2) - \Vol(N_1)}{\alpha(n-1) \int_0^D (\frac{\sinh K r}{K})^{n-1} dr}.
   \label{eqn: volume1}
\end{align}
   \label{lemma: omegacontrol}
\end{lemma}

\begin{figure}
 \begin{center}
  \psfrag{A}[c][c]{$\Omega$}
  \psfrag{B}[c][c]{$N_1$}
  \psfrag{C}[c][c]{$N_2$}
  \psfrag{D}[c][c]{$O_p$}
  \psfrag{E}[c][c]{$P$}
  \includegraphics[width=0.8 \columnwidth]{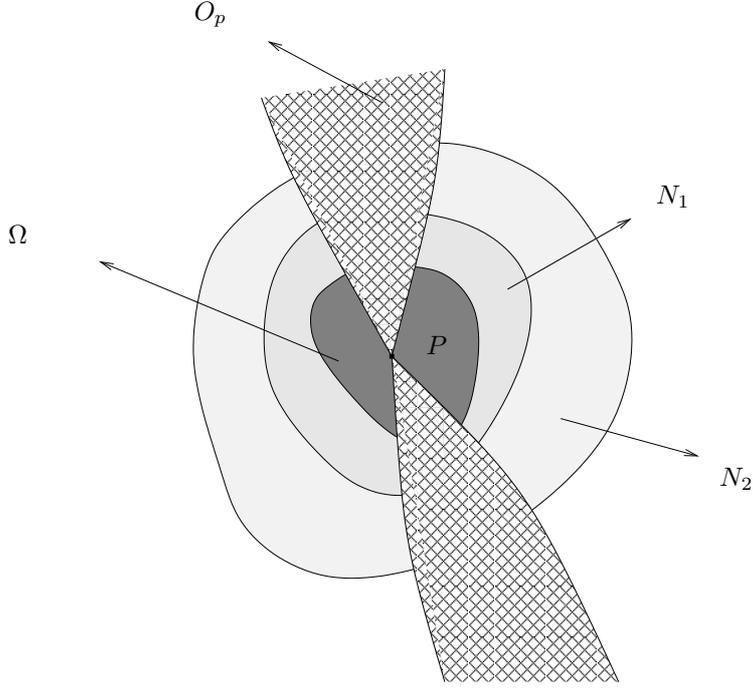}
  \caption{the relation among sets}
  \label{fig: sets}
 \end{center}
 \end{figure}

\begin{proof}
 Choose $p \in \Omega$. Then $(\Omega, \partial \Omega,g)$ is a smooth Riemannian manifold with boundary.
 We  look $(\Omega, \partial \Omega, g)$ as $(N, \partial N,g)$ in our previous argument.   Let
 \[
    O_p \triangleq \{  q \in M| q = \Exp_p {tu}, u  \in \tilde{U}_p \; , t \leq C(u)
    \},
 \]
 where $C(u)$ is the cut radius at direction $u$. Since $u \in
 \tilde{U}_p$, $\tilde{l}(u) = l(u)$. Therefore $M \setminus \Omega  \subset
 O_p$, in particular, $N_2 \setminus N_1 \subset  O_p$. And also we
 know, $N_2 \setminus N_1 \subset N_2 \subset B(p,D)$. Then
\begin{align*}
 \Vol(N_2 \setminus N_1) &\leq \Vol (O_p \cap B(p,D)) \\
        &= \int_{\tilde{U}_p} \int_0^D F(u,r)dr du\\
        &\leq \tilde{\omega}_p \alpha(n-1) \int_0^D (\frac{\sinh K r}{K})^{n-1} dr.
\end{align*}
Consequently,
\begin{align*}
   \tilde{\omega}_p \geq \frac{\Vol(N_2) - \Vol(N_1)}{\alpha(n-1) \int_0^D (\frac{\sinh K r}{K})^{n-1}
   dr}.
\end{align*}
Since $p$ is an arbitrary point in $\Omega$, we have
\begin{align*}
 \tilde{\omega} & = \inf_{p \in \Omega} \tilde{\omega}_p
 \geq \frac{\Vol(N_2) - \Vol(N_1)}{\alpha(n-1) \int_0^D (\frac{\sinh K r}{K})^{n-1}
   dr}.
\end{align*}
\end{proof}

\begin{theorem}
 Suppose $\{(M^n,g(t)), 0 \leq t \leq 1 \}, n \geq 3$ is a Ricci flow solution.
$p \in M$,  and
\begin{align*}
  & Ric(x,t) \geq -(n-1), \; \forall \; (x,t) \in  M \times [0,1];\\
  & Ric(x,t) \leq (n-1), \; \forall \; (x,t) \in B_{g(1)}(p,1) \times [0,1];\\
  & \Vol_{g(1)}(B_{g(1)}(p,1)) \geq \kappa.
\end{align*}
Let $r(\kappa)$ be the solution of  $ \int_0^{r(\kappa)} (\sinh
s)^{n-1} ds = \frac{\kappa}{2 \alpha(n-1) e^{2n(n-1)}}$. Then there
is a uniform Sobolev constant
 $\sigma(n,\kappa)$
for $B_{g(1)}(p,r(\kappa))$ on each time slice, i.e., for any $f \in
W_0^{1,2}(B_{g(1)}(p,r(\kappa)))$,
\begin{align}
   \norm{f}{\frac{2n}{n-2}, B_{g(1)}(p,r(\kappa))}^2
  \leq \sigma(n,\kappa) \norm{\nabla f}{2,B_{g(1)}(p,r(\kappa))}^2.
\label{eqn: localsobolev}
\end{align}
\label{thm: sobolev}
\end{theorem}

\begin{proof}
   Let $N_1 \triangleq B_{g(1)}(p,r(\kappa))$, $N_2 \triangleq B_{g(1)}(p,1)$. Calculating the evolution equation for volume:
\begin{align*}
    \frac{d \Vol_{g(t)}(N_2)}{d t} &=- \int_{N_2} R d\mu \\
         & \leq n(n-1) \Vol_{g(t)}(N_2),   \quad (Ric \geq -(n-1))
\end{align*}
hence,
\begin{align}
 \Vol_{g(t)}(N_2) &\geq e^{n(n-1)(t-1)} \Vol_{g(1)}(N_2) \notag \\
                  &\geq e^{-n(n-1)} \Vol_{g(1)}(N_2) \quad   (0 \leq t \leq 1) \notag \\
                  &\geq e^{-n(n-1)} \kappa. \label{eqn: volumelower}
\end{align}
Similarly, by the condition $Ric \leq (n-1)$,
\begin{align}
 \Vol_{g(t)}(N_1) &\leq e^{n(n-1)(1-t)} \Vol_{g(1)}(N_1) \notag\\
                  &\leq e^{n(n-1)} \Vol_{g(1)}(N_1) \notag\\
                  &=e^{n(n-1)} \int_{B_{g(1)}(p,r(\kappa))} d\mu \notag\\
                  &\leq e^{n(n-1)} \alpha(n-1) \int_0^{r(\kappa)} (\sinh s)^{n-1} ds \notag\\
                  &\leq \frac{ \kappa}{2} e^{-n(n-1)}. \label{eqn: volumeup}
\end{align}
Now we consider the diameter change under Ricci flow. Suppose $\{
\gamma(s), 0 \leq s \leq \rho \}$ is a normalized shortest geodesic
contained in $N_2$ at time $t$, then
\begin{align*}
 \frac{d L_{g(t)}(\gamma)}{dt} &= - \int_0^\rho  Ric(\gamma', \gamma') ds\\
            & \geq -(n-1) L_{g(t)}(\gamma).
\end{align*}
Let $D(t)$ be the diameter of $N_2$ at time $t$, we have
\begin{align*}
 \frac{d^- D(t)}{dt} & \geq -(n-1) D(t),
\end{align*}
hence,
\begin{align}
 D(t) \leq D(1) e^{(n-1)(1-t)} \leq 2 e^{(n-1)}.
 \label{eqn: diameterup}
\end{align}
Choose an arbitrary domain $\Omega \Subset N_1$ with smooth
boundary.By inequalities (\ref{eqn: volumelower}),(\ref{eqn:
volumeup}) and (\ref{eqn: diameterup}), from lemma \ref{lemma:
omegacontrol}, we know
\begin{align*}
 \tilde{\omega}_{g(t)}(\Omega)
   &\geq \frac{\Vol_{g(t)}(N_2) - \Vol_{g(t)}(N_1)}{\alpha(n-1) \int_0^{D(t)} (\sinh s)^{n-1} ds}\\
   &\geq \frac{\kappa e^{-n(n-1)}}{2 \alpha(n-1) \int_0^{2 e^{(n-1)}} (\sinh s)^{n-1} ds}\\
   &\triangleq C_3(n,\kappa).
\end{align*}
Then, from lemma~\ref{lemma: psicontrol}, we have
\begin{align*}
 \frac{\Area(\partial \Omega)^n}{\Vol(\Omega)^{n-1}}
   &\geq C_2(n){C_3(n,\kappa)}^{n+1}\\
   &\triangleq C_4(n,\kappa).
\end{align*}
Since we can approximate any domain by domains with smooth boundary,
we actually get
\begin{align}
 \Phi(N_1) = \inf_{\Omega \Subset N_1}
\frac{\Area(\partial{\Omega})^{n}}{\Vol(\Omega)^{n-1}} \geq
C_4(n,\kappa). \label{eqn: isoperimetric}
\end{align}

Accordingly, by the equivalence of isoperimetric constant and
Sobolev constant, for any $f \in W_0^{1,1}(N_1)$,
\begin{align}
   C_4(n,\kappa) (\int_{N_1} |f|^{\frac{n}{n-1}} d\mu)^{\frac{n-1}{n}}
   \leq \int_{N_1} |\nabla f |.
   \label{eqn: L1sobolev}
\end{align}
We refer the readers to~\cite{SY} for a detailed proof for the
equivalence of inequality (\ref{eqn: isoperimetric}) and inequality
(\ref{eqn: L1sobolev}).
 Let $\gamma >0$, then
\begin{align*}
  \norm{|f|^\gamma}{\frac{n}{n-1},N_1} &\leq \frac{1}{C_4}\norm{\gamma |f|^{\gamma-1} \nabla f}{1,N_1}\\
        &\leq \frac{1}{C_4} \gamma \norm{f^{\gamma-1}}{\frac{p}{p-1},N_1}
                \norm{\nabla f}{p,N_1} \; ,
\end{align*}
therefore,
\begin{align*}
  \norm{f}{\frac{n \gamma}{n-1},N_1}^\gamma \leq \frac{1}{C_4} \norm{f}{\frac{(\gamma-1)p}{p-1},N_1}^{\gamma-1} \cdot \norm{\nabla f}{p,N_1} \; .
\end{align*}
Choose $\gamma = \frac{p(n-1)}{n-p}$, we have
\begin{align*}
  \norm{f}{\frac{np}{n-p},N_1} \leq \frac{1}{C_4} \cdot \frac{p(n-1)}{n-p} \cdot \norm{\nabla f}{p,N_1} \;.
\end{align*}
In particular, choose $ p=2$, let
\begin{align*}
    \sigma(n,\kappa)= (\frac{2(n-1)}{C_4(n,\kappa)(n-2)})^2,
\end{align*}
we obtain
\begin{align*}
    \norm{f}{\frac{2n}{n-2},N_1}^2 \leq
         \sigma(n,\kappa) \norm{\nabla f}{2,N_1}^2
\end{align*}
for any $f \in W_0^{1,2}(N_1)$.
\end{proof}

After we get the local Sobolev constant control, we are able to get
some Moser iteration formula under Ricci flow.

\section{ Moser Iteration of Scalar curvature ($n \geq 3$)}
We will give a detailed construction of local Moser iteration under
Ricci flow in this section. The idea comes from the Moser iteration
in ~\cite{CT}. Let us fix notation first.
\begin{definition}
$\{ (M^n,g(t)), \; 0 \leq t \leq 1\}$ is a  closed Ricci flow
solution. Fixing $p \in M$, $r>0$, we define
\begin{align*}
 &\Omega \triangleq B_{g(1)}(p,r),
 &\Omega' \triangleq B_{g(1)}(p,\frac{r}{2}),\\
 &D \triangleq \Omega \times [0,1],
 &D' \triangleq \Omega' \times [\frac12,1].
\end{align*}
\end{definition}

  Inequality (\ref{eqn: sobolev}) is only  Sobolev inequality for
  time slices. In order to apply Moser iteration on the parabolic
  domain $D$, we need a parabolic version of Sobolev inequality.

\begin{property}
  Suppose there is a uniform Soblev constant $\sigma$  for $\Omega$ at each
time slice, $v \in C^1(D)$, and $v(\cdot,t) \in C_0^1(D),\; \forall
t\in  [0,1]$, we have
 \begin{align}
  \int_D v^{\frac{2(n+2)}{n}} d\mu dt
  \leq \sigma \max_{0 \leq t \leq 1} \norm{v(\cdot,t)}{2,\Omega}^{\frac{4}{n}}
  \int_D |\nabla v|^2  d\mu dt.   \label{eqn: Sobolev}
\end{align}
\end{property}
\begin{proof}
 By H\"{o}lder inequality and inequality (\ref{eqn: sobolev}), we have
\begin{align*}
  \int_D v^{\frac{2(n+2)}{n}} d\mu dt &= \int_0^1 dt
            \int_{\Omega} v^2 \cdot v^{\frac{4}{n}} d\mu \\
      &\leq \int_0^1 dt (\int_\Omega v^{\frac{2n}{n-2}} d\mu)^{\frac{n-2}{n}} \cdot
                   (\int_\Omega v^{\frac{4}{n} \cdot \frac{n}{2}} d\mu)^{\frac{2}{n}}  \\
      &=\int_0^1 \norm{v(\cdot,t)}{2,\Omega}^{\frac{4}{n}}
               (\int_\Omega v^{\frac{2n}{n-2}}d\mu)^{\frac{n-2}{n}} dt  \\
      &\leq \sigma \max_{0 \leq t \leq 1} \norm{v(\cdot,t)}{2,\Omega}^{\frac{4}{n}}
             \int_D |\nabla v|^2  d\mu dt.
\end{align*}
\end{proof}

Then we start the main Lemmas in this section.
\begin{lemma}
  $\{ (M^n,g(t)), \; 0 \leq t \leq 1\}$ is a closed Ricci flow solution with $Ric \geq
  -B$. Suppose there is a uniform Soblev constant $\sigma$  for $\Omega$ at each
time slice.   If $u \in C^1(D)$ and $u \geq 0$,
\begin{align}
       \D{u}{t} \leq \Lap u + fu +h,
       \label{eqn: equation1}
\end{align}
in distribution sense, and  $\norm{f}{q,D} + \norm{R_{-}}{q,D} +1
\leq C_0$ for some $q
> \frac{n}{2} +1$. Then there is  a constant $C_a=
C_a(n,q,\sigma,C_0,r,B)$ such that
\begin{align}
   \norm{u}{\infty, D'} \leq C_a (\norm{u}{\frac{n+2}{n},D} + \norm{h}{q,D} \cdot \norm{1}{\frac{n+2}{n},D}).
   \label{eqn: control0}
\end{align}
\label{lemma: iteration1}
\end{lemma}

\begin{proof}
 Choose a cutoff function $\eta \in C^{\infty}(D)$  such that
$\eta(\cdot,t) \in C_0^\infty(\Omega), \; \forall t\in [0,1]$, and
$\eta(x,0) \equiv 0$. Moreover, $\eta(x,\cdot)$ is a nondecreasing
function for every $x \in \Omega$.

\begin{figure}
 \begin{center}
  \psfrag{A}[c][c]{$\Omega$}
  \psfrag{B}[c][c]{$\supp(\eta)$}
  \includegraphics[width=0.5 \columnwidth]{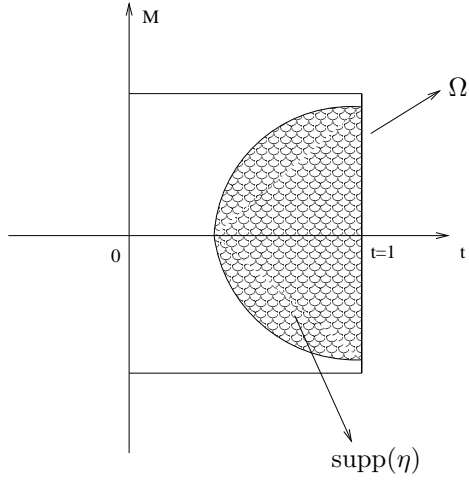}
  \caption{a cutoff function}
  \label{fig: eta}
 \end{center}
\end{figure}

 Define
 \begin{align*}
 \kappa \triangleq  \norm{h}{q,D} , \quad
  v \triangleq u+ \kappa.
 \end{align*}

 Fix $\beta >1$, use $\eta^2 (u+\kappa)^{\beta-1}$ as a test function, from inequality (\ref{eqn: equation1}),
\begin{align*}
     -\Lap v+ \D{v}{t} \leq fu +h.
\end{align*}
Then, for any $s \in (0,1]$,we have
\begin{align*}
  &\quad \int_0^s \int_\Omega (-\Lap v)\eta^2 v^{\beta-1} d\mu dt
 +\int_0^s \int_\Omega \D{v}{t} \eta^2 v^{\beta-1} d\mu dt\\
 &\leq    \int_0^s \int_\Omega (fu+h)\eta^2 (u+\kappa)^{\beta-1} d\mu
 dt\\
 &\leq \int_0^s \int_\Omega (|f|+\frac{|h|}{\kappa}) \eta^2
 v^{\beta} d\mu dt.
\end{align*}

Note that $\D{d\mu}{t} = -R d\mu$, integrating by parts yields
\begin{align}
   &\int_0^s  \int_\Omega(2 \eta <\nabla \eta, \nabla v> v^{\beta-1}
   +(\beta-1) \eta^2 v^{\beta-2}|\nabla v|^2)d\mu dt \notag \\
   &\quad +\frac{1}{\beta}(\int_\Omega \eta^2 v^\beta d\mu|_s
       - \int_0^s\int_\Omega 2 \eta \D{\eta}{t}v^\beta d\mu dt + \int_0^s  \int_\Omega \eta^2 v^\beta R d\mu dt) \notag \\
   &\leq \int_0^s \int_\Omega (|f|+\frac{|h|}{\kappa}) \eta^2
   v^{\beta} d\mu dt.
    \label{eqn: ipart}
\end{align}
By Schwartz inequality,
\begin{align}
  \int_0^s \int_\Omega 2 \eta <\nabla \eta, \nabla v> v^{\beta-1} d\mu dt
 \geq -\epsilon^2 \int_0^s \int_\Omega \eta^2v^{\beta-2} |\nabla v|^2
  - \frac{1}{\epsilon^2} \int_0^s \int_\Omega v^\beta |\nabla \eta|^2.
\label{eqn: Schwartz}
\end{align}
Plugging inequality (\ref{eqn: Schwartz}) into (\ref{eqn: ipart}),
we get
\begin{align*}
  &(\beta-1-\epsilon^2) \int_0^s \int_\Omega \eta^2 v^{\beta-2}|\nabla v|^2 d\mu dt
      + \frac{1}{\beta} \int_\Omega \eta^2 v^\beta d\mu|_s \\
  &\quad \leq \int_0^s \int_\Omega (|f|+\frac{|h|}{\kappa}) \eta^2 v^\beta d\mu dt
    + \frac{1}{\epsilon^2} \int_0^s \int_\Omega v^\beta |\nabla \eta|^2 d\mu dt\\
  &\quad \quad +\frac{1}{\beta} (\int_0^s\int_\Omega 2\eta \D{\eta}{t} v^\beta d\mu dt
                    -\int_0^s \int_\Omega \eta^2 v^\beta R d\mu dt).
\end{align*}
Let $\epsilon^2 = \frac{\beta-1}{2}$, since $|\nabla
v^{\frac{\beta}{2}}|^2= \frac{\beta^2}{4} v^{\beta-2} |\nabla v|^2$,
we know
\begin{align*}
   & 2(1-\frac 1 \beta) \int_0^s \int_\Omega \eta^2 |\nabla v^{\frac{\beta}{2}}|^2 d\mu dt
  + \int_\Omega \eta^2 v^\beta d\mu|_s \\
   & \quad \leq \beta \int_0^s \int_\Omega (|f|+\frac{|h|}{\kappa}+R_{-}) \eta^2 v^\beta d\mu
   dt\\
   &\quad \quad +\frac{2\beta}{\beta-1}\int_0^s \int_\Omega v^\beta |\nabla \eta|^2 d\mu dt
     +\int_0^s\int_\Omega 2\eta \D{\eta}{t} v^\beta d\mu dt.
\end{align*}
Since
\[
|\nabla (\eta v^{\frac{\beta}{2}})|^2
 \leq 2 \eta^2 |\nabla v^{\frac{\beta}{2}}|^2+2v^\beta|\nabla \eta|^2,
\]
we have
\begin{align*}
  & (1-\frac 1 \beta) \int_0^s \int_\Omega  |\nabla (\eta v^{\frac{\beta}{2}})|^2 d\mu dt
  + \int_\Omega \eta^2 v^\beta d\mu|_s \\
   &\quad \leq \beta \int_0^s \int_\Omega (|f|+\frac{|h|}{\kappa}+R_{-}) \eta^2 v^\beta d\mu
   dt\\
   &\quad \quad  +2(\frac{\beta}{\beta-1}+\frac{\beta-1}{\beta})\int_0^s \int_\Omega v^\beta |\nabla \eta|^2 d\mu dt
   +\int_0^s\int_\Omega 2\eta \D{\eta}{t} v^\beta d\mu dt.
\end{align*}
Therefore,
\begin{align*}
  &\int_0^s \int_\Omega  |\nabla (\eta v^{\frac{\beta}{2}})|^2 d\mu dt
  +\int_\Omega \eta^2 v^\beta d\mu|_s \\
  & \quad \leq \Lambda(\beta) (\int_0^s \int_\Omega   (|f|+\frac{|h|}{\kappa}+R_{-}) \eta^2 v^\beta d\mu dt
     +\int_0^s \int_\Omega v^\beta |\nabla \eta|^2 d\mu dt
     +\int_0^s\int_\Omega 2\eta \D{\eta}{t} v^\beta d\mu dt)\\
  & \quad \leq \Lambda(\beta) ((\int_0^s \int_\Omega (|f|+\frac{|h|}{\kappa}+R_{-})^q d\mu dt )^{\frac1q}
   (\int_0^s \int_\Omega (\eta^2 v^\beta)^{\frac{q}{q-1}} d\mu dt)^{\frac{q-1}{q}}\\
  & \quad \quad   +\int_0^s \int_\Omega v^\beta |\nabla \eta|^2 d\mu dt
     +\int_0^s \int_\Omega 2\eta \D{\eta}{t} v^\beta d\mu dt)\\
  & \quad \leq \Lambda(\beta)
   \{(\norm{f}{q,D} + \norm{R_{-}}{q,D} + 1) (\int_0^s \int_\Omega (\eta^2 v^\beta)^{\frac{q}{q-1}} d\mu
   dt)^{\frac{q-1}{q}}\\
  &\quad \quad +\int_0^s \int_\Omega v^\beta |\nabla \eta|^2 d\mu dt
     +\int_0^s \int_\Omega 2\eta \D{\eta}{t} v^\beta d\mu dt \}\\
  & \quad \leq \Lambda(\beta)
   (C_0 (\int_0^s \int_\Omega (\eta^2 v^\beta)^{\frac{q}{q-1}} d\mu dt)^{\frac{q-1}{q}}
    +\int_0^s \int_\Omega v^\beta |\nabla \eta|^2 d\mu dt
     +\int_0^s \int_\Omega 2\eta \D{\eta}{t} v^\beta d\mu dt).
\end{align*}
We can choose $\Lambda(\beta)= 6\beta$ if $\beta \geq 2$. In
particular,
\begin{align*}
  &\max_{0 \leq s \leq 1} \int_\Omega \eta^2 v^\beta d\mu|_s
     \leq  \Lambda(\beta) (\norm{(|\nabla \eta|^2+ 2 \eta \D{\eta}{t} )v^\beta }{1,D}
     + C_0 \norm{\eta^2 v^\beta}{\frac{q}{q-1},D}),\\
  &\int_D \eta^2 |\nabla v^{\frac{\beta}{2}}|^2 d\mu dt
     \leq  \Lambda(\beta) (\norm{(|\nabla \eta|^2+ 2 \eta \D{\eta}{t} )v^\beta }{1,D}
     + C_0 \norm{\eta^2 v^\beta}{\frac{q}{q-1},D}).
\end{align*}
The Sobolev inequality (\ref{eqn: Sobolev}) on the parabolic domain
$D$ yields
\begin{align}
   \norm{\eta^2 v^\beta}{\frac{n+2}{n},D} \leq
   \sigma^{\frac{n}{n+2}}
   \Lambda(\beta) (\norm{(|\nabla \eta|^2+ 2 \eta \D{\eta}{t} )v^\beta }{1,D} + C_0 \norm{\eta^2
                           v^\beta}{\frac{q}{q-1},D}).
    \label{eqn: normes}
\end{align}
Since $q > \frac{n+2}{2}$, $\frac{q}{q-1} \leq \frac{n+2}{n}$, by
interpolation inequality,
\[
   \norm{\eta^2 v^\beta}{\frac{q}{q-1},D} \leq \epsilon' \norm{\eta^2
   v^\beta}{\frac{n+2}{n},D} + (\epsilon')^{-\nu} \norm{\eta^2 v^\beta}{1,D},
\]
where $\nu = \frac{n+2}{2q-n-2}$. Therefore,
\begin{align*}
  &(1-\Lambda(\beta) \sigma^{\frac{n}{n+2}}C_0 \epsilon') \norm{\eta^2
  v^\beta}{\frac{n+2}{n},D}\\
  &\quad \quad  \leq \Lambda(\beta)\sigma^{\frac{n}{n+2}}(\norm{(|\nabla \eta|^2+ 2 \eta \D{\eta}{t} )v^\beta }{1,D} + C_0 \cdot (\epsilon')^{-\nu} \norm{\eta^2 v^\beta}{1,D}).
\end{align*}

 Let $\epsilon' =
\frac{1}{2\Lambda(\beta)\sigma^{\frac{n}{n+2}} C_0}$, we get
\begin{align*}
  \norm{\eta^2 v^\beta}{\frac{n+2}{n},D} \leq 2\Lambda(\beta)\sigma^{\frac{n}{n+2}} (\norm{(|\nabla \eta|^2+ 2 \eta \D{\eta}{t} )v^\beta }{1,D} +
  C_0 \cdot (2\Lambda(\beta)\sigma^{\frac{n}{n+2}} C_0)^\nu \norm{\eta^2 v^\beta}{1,D}).
\end{align*}
Since we can always choose $\Lambda(\beta) \geq 1$, we obtain
\begin{align}
  \norm{\eta^2 v^\beta}{\frac{n+2}{n},D} \leq C_1(n,q,\sigma,C_0)
  \Lambda(\beta)^{1+\nu} \int_D ( |\nabla \eta|^2 +2\eta \D{\eta}{t} + \eta^2) v^\beta
    d\mu dt.
  \label{eqn: preiteration}
\end{align}

\begin{figure}
 \begin{center}
  \psfrag{A}[c][c]{$D$}
  \psfrag{B}[c][c]{$D_1$}
  \psfrag{C}[c][c]{$D_2$}
  \psfrag{D}[c][c]{$D'$}
  \psfrag{E}[c][c]{$(p,0)$}
  \psfrag{F}[c][c]{$\Omega$}
  \psfrag{G}[c][c]{$\Omega_1$}
  \psfrag{H}[c][c]{$\Omega_2$}
  \psfrag{I}[c][c]{$\Omega'$}
  \includegraphics[width=0.8 \columnwidth]{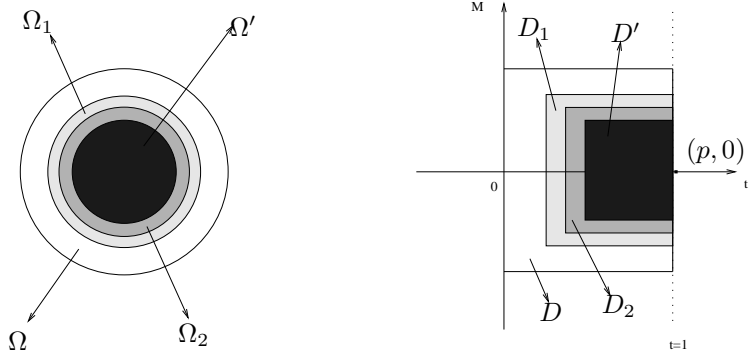}
  \caption{the sequence of domains}
  \label{fig: domains}
 \end{center}
 \end{figure}

Then we construct cutoff functions and domains. Define
\[
  t_k \triangleq \frac12 - \frac{1}{2^{k+1}},
 \quad r_k \triangleq (\frac12 + \frac{1}{2^{k+1}})r, \quad k \geq 0,
\]
\begin{align}
  \Omega_k \triangleq B_{g(1)}(p,r_k), D_k \triangleq \Omega_k\times [t_k,1], \quad k \geq 0.
  \label{eqn: domains}
\end{align}
Let $\gamma \in C^{\infty}(\R,\R)$, $0 \leq \gamma' \leq 2 $, and
\[
    \gamma(t)= \left \{
    \begin{array}{ll}
    &0, \quad  t \leq  0,\\
    &1, \quad  t \geq  1.\\
    \end{array}
    \right.
\]
Define $\gamma_k(t) \triangleq \gamma(\frac{t-t_{k-1}}{t_k
-t_{k-1}}), \quad k \geq 1 $.

Let $\rho \in C^{\infty}(\R,\R)$, $-2 \leq \rho' \leq 0$, and
\[
    \rho(s)= \left \{
    \begin{array}{ll}
    &1, \quad  s \leq  0,\\
    &0, \quad  s \geq  1.\\
    \end{array}
    \right.
\]
Define $\rho_k(s) \triangleq \rho(\frac{s-r_k}{r_{k-1}-r_k}),\quad k
\geq 1$. Then let
\[
    \eta_k(x,t) = \gamma_k(t) \rho_k(d_{g(1)}(x,p)).
\]
Therefore,  $0 \leq \eta_k \leq 1$, and
\[
  \eta_k(x,t)= \left \{
    \begin{array}{ll}
    &0, \quad  (x,t) \in  D \slash D_{k-1},\\
    &1, \quad  (x,t) \in D_k.\\
    \end{array}
    \right.
\]

\begin{figure}
 \begin{center}
  \psfrag{A}[c][c]{$y= \gamma(t)$}
  \psfrag{B}[c][c]{$y= \rho(s)$}
  \includegraphics[width=0.8 \columnwidth]{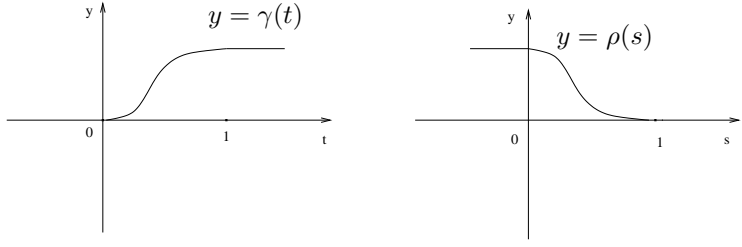}
  \caption{basic cutoff functions}
  \label{fig: basic}
 \end{center}
 \end{figure}

Moreover,
\begin{align*}
  |\D{\eta_k}{t}| = |\D{\gamma_k(t)}{t} \rho_k(r(x))|
 = |\frac{\gamma'}{t_k-t_{k-1}}\rho_k(d_{g(1)}(x,p))|
\leq 2^{k+2},
\end{align*}
\begin{align*}
 |\nabla \eta_k|_{g(1)} &= |\gamma_k(t) \nabla \rho_k(d_{g(1)}(x,p))|_{g(1)}\\
     &=|\gamma_k(t) \rho_k'(d_{g(1)}(r,p)) \nabla d_{g(1)}(x,p)|_{g(1)}\\
     &\leq |\rho_k'(d_{g(1)}(x,p))|_{g(1)}\\
     &\leq \frac{\rho'}{r_{k-1}-r_k}  \leq 2^{k+2} r^{-1}.
\end{align*}
Note that
\begin{align*}
   \frac{d}{dt} |\nabla \eta_k|_{g(t)}^2 = 2 Ric_{g(t)}(\nabla \eta_k, \nabla
   \eta_k) \geq -2B |\nabla \eta_k|_{g(t)}^2 ,
\end{align*}
hence
\begin{align*}
  |\nabla \eta_k|_{g(t)}^2 \leq e^{2B(1-t)} |\nabla
  \eta_k|_{g(1)}^2
  \leq  e^{2B}|\nabla \eta_k|_{g(1)}^2.
\end{align*}
 Therefore,we know
\begin{align*}
&|\D{\eta_k}{t}| \leq 2^{k+2},
\end{align*}
\begin{align}
|\nabla \eta_k|_{g(t)} \leq e^B 2^{k+2}r^{-1}, \quad \forall t \in
[0,1]. \label{eqn: cutoff}
\end{align}

If $\beta \geq 2$,$\Lambda(\beta)=6\beta$, by inequality (\ref{eqn:
preiteration}), we have
\begin{align*}
  \norm{v^\beta}{\frac{n+2}{n},D_k}
  &=\norm{{\eta_k}^2 v^\beta}{\frac{n+2}{n}, D_k} \\
  &\leq \norm{{\eta_k}^2 v^\beta}{\frac{n+2}{n}, D_{k-1}}\\
  &\leq C_2(n,q,\sigma,
  C_0) \beta^{1+\nu} \int_{D_{k-1}}(|\nabla \eta_k|^2 + 2 \eta_k \D{\eta_k}{t}+\eta_k^2) v^\beta d\mu dt\\
 &\leq 4^{k+2}C_3(r,B)C_2(n,q,\sigma,C_0)\beta^{1+\nu}\int_{D_{k-1}}v^\beta d\mu dt\\
 &\triangleq C_4(n,q,\sigma,C_0,r,B)\cdot 4^{k-1}\cdot \beta^{1+\nu} \norm{v^\beta}{1,D_{k-1}},
\end{align*}
consequently,
\begin{align}
  \norm{v}{\frac{n+2}{n}\beta,D_k} \leq C_4^{\frac{1}{\beta}}\cdot 4^{\frac{k-1}{\beta}}
           \cdot \beta^{\frac{1+\nu}{\beta}} \norm{v}{\beta,D_{k-1}}  \; .
  \label{eqn: iteration}
\end{align}
Let $\lambda \triangleq \frac{n+2}{n}$, then
\begin{align*}
  \norm{v}{\lambda^k,D_k}
 &\leq C_4^{\frac{1}{\lambda^{k-1}}+ \frac{1}{\lambda^{k-2}}+\cdots +\frac{1}{\lambda^{k_0}}}
  4^{\frac{k-1}{\lambda^{k-1}}+\cdots + \frac{k_0}{\lambda^{k_0}}}
  \lambda^{(1+\nu)(\frac{k-1}{\lambda^{k-1}} +\cdots + \frac{k_0}{\lambda^{k_0}})}
  \norm{v}{\lambda^{k_0},D_{k_0}}\\
 &\triangleq C_5(n,q,\sigma,C_0,r,B) \norm{v}{\lambda^{k_0},D_{k_0}}.
\end{align*}
Here $k_0=k_0(n)$ is the smallest integer such that $\lambda^{k_0}
\geq 2$.  If $\beta<2$, since (\ref{eqn: preiteration}) is  true, we
can still do iteration. Starting from $\norm{v}{\lambda,D_1}$, in
$k_0$ steps, we can get a control of
$\norm{v}{\lambda^{k_0},D_{k_0}}$. That is,
\begin{align*}
   \norm{v}{\lambda^{k_0},D_{k_0}} \leq C_6(n,q,\sigma,C_0,r,B)\norm{v}{\lambda,D_1}.
\end{align*}
Consequently,
\begin{align}
   \norm{v}{\lambda^{k}, D_k} \leq C_7(n,q,\sigma,C_0,r,B) \norm{v}{\lambda,D_1}, \quad \forall \;  k \geq 0.
   \label{eqn: control1}
\end{align}
Actually, what we get is
\begin{align}
  \norm{v}{\lambda^{k_2},D_{k_2}} \leq C_7(n,q,\sigma,C_0,r,B)
  \norm{v}{\lambda^{k_1}, D_{k_1}}, \quad \forall \; 0 \leq k_1 \leq k_2.
 \label{eqn: control2}
\end{align}
From inequality (\ref{eqn: control1}), and $D' \subset D_k, \forall
\; k \geq 0$, we get
\begin{align*}
  \norm{v}{\lambda^k,D'} \leq \norm{v}{\lambda^k,D_k} \leq C_7 \norm{v}{\lambda, D_1} \leq C_7 \norm{v}{\lambda, D}.
\end{align*}
Let $k \to \infty$, $C_a \triangleq C_7(n,q,\sigma,C_0,r,B)$, we get
\begin{align*}
\norm{v}{\infty,D'} \leq C_a(n,q,\sigma,C_0,r,B)\norm{v}{\lambda, D}
\;.
\end{align*}
Since $u \geq 0$, we have
\begin{align*}
 \norm{u}{\infty,D'} & \leq \norm{v}{\infty,D'}\\
    & \leq C_a(n,q,\sigma,C_0,r,B)\norm{v}{\lambda, D}\\
    &\leq C_a(n,q,\sigma,C_0,r,B)(\norm{u}{\lambda, D} +\kappa \norm{1}{\lambda, D})\\
    &= C_a(n,q,\sigma,C_0,r,B)(\norm{u}{\lambda,D} + \norm{h}{q,D}
    \norm{1}{\lambda,D}).
\end{align*}
\end{proof}

\begin{remark}
 From our proof, in order inequality (\ref{eqn: control2}) to be true, we only need
 $ \norm{f}{q,D_{k_1}} + \norm{R_{-}}{q,D_{k_1}} +1
\leq C_0$. Consequently, inequality (\ref{eqn: control0}) is true
for the same constant if
 $D$ is replaced by $D_k$, i.e.,
\begin{align}
   \norm{u}{\infty,D'} \leq  C_a(n,\sigma,C_0,r,B)(\norm{u}{\lambda^k,D_k} + \norm{h}{q,D}
   \norm{1}{\lambda^k,D_k}).
   \label{eqn: control0'}
\end{align}
\label{remark: gie}
\end{remark}

\begin{lemma}
  $\{ (M^n,g(t)), \; 0 \leq t \leq 1 \}$ is a closed Ricci flow solution with $Ric \geq -B$.
 There is a uniform Soblev constant $\sigma$  for $\Omega$ at each time slice.
 If $u \in C^1(D)$ and $u \geq 0$,
\begin{align*}
       \D{u}{t} \leq \Lap u + fu +h,
\end{align*}
in distribution sense.  Here $ f \in L^{\frac{n+2}{2}}(D)$. Fix
$\beta>1$.  Then there are two constants
$\;\delta_b(n,\sigma,\beta), C_b(n, \sigma,r,B,\beta)$ such that if
\; $\norm{f}{\frac{n+2}{2},D}+ \norm{R_{-}}{\frac{n+2}{2},D} \leq
\delta_b$, then
\begin{align}
 \norm{u}{\frac{n+2}{n}\beta, D_1} \leq C_b(n,\sigma,r,B,\beta)( \norm{u}{\beta,D} +  \norm{h}{\frac{n+2}{2},D}
 \norm{1}{\beta,D}) \; .
 \label{eqn: criticalcontrol}
\end{align}
Here $D_1$ is defined by equation (\ref{eqn: domains}).
\label{lemma: iteration2}
\end{lemma}
\begin{proof}
  Let $\eta= \eta_1$, then we do the calculation as in the proof of lemma~\ref{lemma: iteration1}.
  Instead of $\kappa = \norm{h}{q,D}$ in the previous lemma, we let $\kappa = l \cdot
  \norm{h}{\frac{n+2}{2},D}$ for some positive number $l$.
  We can get a similar inequality as inequality (\ref{eqn: normes}),
\begin{align*}
   \norm{\eta_1^2 v^\beta}{\frac{n+2}{n},D} &\leq
   \sigma^{\frac{n}{n+2}} \Lambda(\beta) \{ \int_D (|\nabla \eta_1|^2+ 2 \eta_1 \D{\eta_1}{t} )v^\beta d\mu dt +\\
   & \quad \quad (\norm{f}{\frac{n+2}{2},D} +
\norm{R_{-}}{\frac{n+2}{2},D} +\frac{1}{l})\norm{\eta_1^2
v^\beta}{\frac{n+2}{n},D} \}.
\end{align*}
 If $\norm{f}{\frac{n+2}{2},D}+\norm{R_{-}}{\frac{n+2}{2},D} \leq \frac{1}{4\sigma^{\frac{n}{n+2}}
\Lambda(\beta)}$, choose $l=4\sigma^{\frac{n}{n+2}}
\Lambda(\beta)+1$,
 we obtain
\begin{align*}
  \norm{\eta_1^2 v^\beta}{\frac{n+2}{n},D}
&\leq  2\sigma^{\frac{n}{n+2}} \Lambda(\beta) \int_D (|\nabla
\eta_1|^2+ 2
\eta_1 \D{\eta_1}{t}) v^\beta d\mu dt\\
&\leq  2\sigma^{\frac{n}{n+2}} \Lambda(\beta)
C_8(r,B)\norm{v^\beta}{1,D}\; .
\end{align*}
Consequently,
\begin{align*}
 \norm{v}{\frac{n+2}{n}\beta, D_1}^\beta &=\norm{v^\beta}{\frac{n+2}{n}, D_1}\\
         &\leq \norm{\eta_1^2 v^\beta}{\frac{n+2}{n},D} \\
         &\leq 2\sigma^{\frac{n}{n+2}} \Lambda(\beta) C_8(r,B) \norm{v^\beta}{1,D}\\
         &=2\sigma^{\frac{n}{n+2}} \Lambda(\beta) C_8(r,B) \norm{v}{\beta,D}^\beta \; .
\end{align*}
Let $C_9 (n,\sigma,r,B,\beta)\triangleq (2\sigma^{\frac{n}{n+2}}
\Lambda(\beta) C_8(r,B))^{\frac{1}{\beta}}$, we get
\[
 \norm{v}{\frac{n+2}{n}\beta, D_1} \leq C_9(n,\sigma,r,B,\beta) \norm{v}{\beta,D} \; .
\]
Since $v=u + \kappa,\; u\geq 0$,
\begin{align*}
 \norm{u}{\frac{n+2}{n}\beta, D_1} &\leq \norm{v}{\beta \cdot \frac{n+2}{n},
 D_1}\\
 &\leq  C_9(n,\sigma,r,B, \beta) \norm{v}{\beta,D}\\
 &\leq  C_9(n,\sigma,r,B, \beta)( \norm{u}{\beta,D} +
 \norm{\kappa}{\beta,D})\\
 &=  C_9(n,\sigma,r,B,\beta)( \norm{u}{\beta,D} + l \cdot \norm{h}{\frac{n+2}{2},D}
 \norm{1}{\beta,D})\\
 &\leq  C_b(n,\sigma,r,B,\beta)( \norm{u}{\beta,D} +  \norm{h}{\frac{n+2}{2},D}
 \norm{1}{\beta,D}).
\end{align*}

Therefore, we finish the proof if we choose
\begin{align*}
   \delta_b(n,\sigma,\beta) &=\frac{1}{4\sigma^{\frac{n}{n+2}} \Lambda(\beta)},\\
   C_b(n,\sigma,r,B,\beta) &=C_9(n,\sigma,r,B,\beta)
     \cdot  (4\sigma^{\frac{n}{n+2}} \Lambda(\beta)+1) .
\end{align*}
\end{proof}

Before we use Moser iteration for $R$, we need some volume control.

\begin{property}
$\{ (M^n,g(t)), \; 0 \leq t \leq 1 \}$ is a closed Ricci flow
solution.
\begin{align*}
  |Ric(x,t)|  \leq (n-1) , \quad \forall (x,t) \in \Omega \times
  [0,1].
\end{align*}
Then there exists a constant $\tilde{V}(n,r) \geq 1$ such that
\begin{align}
   \norm{1}{q, D} \leq \tilde{V}^{\frac{1}{q}} \leq \tilde{V}, \quad
   \forall \; q \geq 1.
   \label{eqn: tildevolume}
\end{align}
\label{property: tildevolume}
\end{property}

\begin{proof}
 Since $\Omega = B_{g(1)}(p,r)$,  $Ric \leq (n-1)$, by the
 evolution of geodesic length under Ricci flow, we have
\begin{align*}
 \Omega \subset B_{g(t)}(p, e^{(n-1)}r), \quad \forall t \in [0,1].
\end{align*}
On the other hand, $Ric \geq -(n-1)$, by volume comparison theorem,
we obtain
\begin{align*}
  \int_{B_{g(t)}(p,e^{(n-1)}r)} d\mu \leq
   \alpha(n-1) \int_0^{e^{(n-1)}r} (\sinh r)^{n-1}dr \triangleq C_{10}(n,r),
\end{align*}
where $\alpha(n-1)$ is the area of $S^{n-1}$ with canonical metric.
Hence,
\begin{align*}
  \norm{1}{1,D} &= \int_0^1 \int_\Omega d\mu dt\\
             &\leq \int_0^1 \int_{B_{g(t)}(p,e^{(n-1)}r)} d\mu dt \\
             &\leq C_{10}.
\end{align*}
Let $\tilde{V}(n,r) \triangleq \max \{C_{10}, 1 \}$, then
\begin{align*}
  \norm{1}{q,D} = \norm{1}{1,D}^{\frac{1}{q}} \leq
  \tilde{V}^{\frac{1}{q}} \leq \tilde{V}, \quad \forall \;q \geq 1.
\end{align*}
\end{proof}

 Now we can apply Moser iteration to $R$.

\begin{theorem}
$\{ (M^n,g(t)), \; 0 \leq t \leq 1 \}$ is a closed Ricci flow
solution. Suppose
\begin{align*}
  Ric(x,t) & \geq -B, \quad \forall (x,t) \in M \times [0,1], \quad  0 \leq B \leq 1,\\
  Ric(x,t) & \leq (n-1) , \quad \forall (x,t) \in \Omega \times
  [0,1].
\end{align*}
There is a uniform Soblev constant $\sigma$  for $\Omega$ at each
time slice. Then there are constants $\delta(n,\sigma,r),
C(n,\sigma,r)$ such that if $\norm{R}{\frac{n+2}{2},D}+B \leq
\delta$, then
\begin{align}
     \norm{R_{+}}{\infty,D'} \leq C(\norm{R}{\frac{n+2}{2},D}+B) \;.
\end{align}
\label{thm: scalarestimate}
\end{theorem}

\begin{proof}
   Since  $Ric \geq -B$, define $\hat{R} \triangleq R+nB$, we get
   inequality (\ref{eqn: rhat}),
  \begin{align*}
   \D{\hat{R}}{t} \leq \Lap \hat{R} +2(\hat{R} -2B )\hat{R}
   +2nB^2.
  \end{align*}

Because  $0 \leq B \leq 1$, in $D= \Omega \times [0,1]$,
 $|Ric| \leq (n-1)$, by Property~\ref{property: tildevolume},
\begin{align*}
  \norm{1}{q,D} = \norm{1}{1,D}^{\frac{1}{q}} \leq
  \tilde{V}^{\frac{1}{q}} \leq \tilde{V}, \quad \forall \;q \geq 1.
\end{align*}

  Let $u=\hat{R}, f=2(\hat{R}-2B), h=2nB^2$. As in lemma ~\ref{lemma:
  iteration2}, let
 \begin{align*}
   \beta &= \frac{n+2}{2}; \\
   \delta_b &=\delta_b(n,\sigma,\beta);  \\
   C_b &=  C_b(n,\sigma,r,1,\beta).
 \end{align*}

  If $\norm{R}{\frac{n+2}{2},D}+B$ is very small,
  say,
\begin{align*}
   \norm{R}{\frac{n+2}{2},D}+B \leq \delta(n,\sigma,r) \triangleq \frac{\delta_b}{3n \tilde{V}},
\end{align*}
 then
 \begin{align*}
   &\quad \norm{2(\hat{R}-2B)}{\frac{n+2}{2},D} + \norm{R_{-}}{\frac{n+2}{2},
   D}\\
   & = \norm{2(R+(n-2)B)}{\frac{n+2}{2},D}+\norm{R_{-}}{\frac{n+2}{2},D} \\
   &\leq 3\norm{R}{\frac{n+2}{2},D} +2(n-2)B \norm{1}{\frac{n+2}{2},D}\\
   &\leq 3n\tilde{V}^{\frac{2}{n+2}} (\norm{R}{\frac{n+2}{2},D}
   +B)\\
   &\leq \frac{\delta_b}{\tilde{V}^{\frac{n}{n+2}}} \\
   &\leq \delta_b,
 \end{align*}
  hence, by lemma~\ref{lemma: iteration2},
\begin{align}
  \norm{\hat{R}}{\frac{n+2}{n} \frac{n+2}{2},D_1} &\leq  C_b (\norm{\hat{R}}{\frac{n+2}{2}, D}
  +     2nB^2 \norm{1}{\frac{n+2}{2},D}^2 ) \notag \\
    & \leq C_b(\norm{R}{\frac{n+2}{2},D}+nB \norm{1}{\frac{n+2}{2},D} +  2nB^2 \norm{1}{\frac{n+2}{2},D}^2) \notag\\
    & \leq C_b(\norm{R}{\frac{n+2}{2},D} +3nB
    \tilde{V}^{\frac{4}{n+2}}) \notag \\
    & \leq C_b 3n \tilde{V}^{\frac{4}{n+2}}
    (\norm{R}{\frac{n+2}{2},D}+B) \label{eqn: control3}\\
    &\leq C_b \delta_b.
  \label{eqn: control4}
\end{align}
  Now let $ q=\frac{n+2}{n}\frac{n+2}{2} >\frac{n+2}{2}$,
   then from inequality (\ref{eqn: control4}),
\begin{align*}
 \norm{2(\hat{R}-2B)}{q,D_1} +\norm{R_{-}}{q,D_1} +1
  &\leq 3\norm{\hat{R}}{q,D_1} +(n+4)B \norm{1}{q,D_1} +1\\
  &\leq 3C_b \delta_b +(n+4)B \tilde{V}^{\frac{1}{q}}+1\\
  &\leq 3C_b \delta_b + \delta_b+1.
\end{align*}
Note that $0 \leq B \leq 1$, by the definition of $C_a$ in
Lemma~\ref{lemma: iteration1}, we get
\begin{align*}
C_a(n,\frac{(n+2)^2}{2n},(3C_b+1)\delta_b+1, \sigma,r,B) \leq
C_a(n,\frac{(n+2)^2}{2n},(3C_b+1)\delta_b+1, \sigma,r,1).
\end{align*}
Let $C_a = C_a(n,\frac{(n+2)^2}{2n},(3C_b+1)\delta_b+1,
\sigma,r,1)$.

  From  Remark~\ref{remark: gie}, we have
\begin{align}
  \norm{\hat{R}}{\infty,D'} &\leq  C_a(n,\frac{(n+2)^2}{2n},(3C_b+1)\delta_b+1, \sigma,
  r,B) (\norm{\hat{R}}{\frac{n+2}{n},D_1} +\norm{h}{q,D} \norm{1}{\frac{n+2}{n},
  D_1}) \notag \\
  &\qquad \qquad \qquad \qquad \qquad \qquad \qquad \qquad \qquad \qquad \textrm{by  H\"{o}lder
  inequality} \notag \\
  &\leq C_a (\norm{\hat{R}}{\frac{(n+2)^2}{2n},D_1} \norm{1}{\frac{n+2}{n},D_1}
  +\norm{2nB^2}{q,D} \norm{1}{\frac{n+2}{n},D_1}) \notag \\
  & \qquad \qquad \qquad \qquad \qquad \qquad \qquad \qquad \qquad \qquad \textrm{from inequality (\ref{eqn:
  control3})} \notag \\
  &\leq C_a ( 3nC_b \tilde{V}^{\frac{4}{n+2}} (\norm{R}{\frac{n+2}{2},D}+B) \norm{1}{\frac{n+2}{n},D_1}
       +2nB^2 \norm{1}{\frac{(n+2)^2}{2n},D}
       \norm{1}{\frac{n+2}{n},D_1}) \notag \\
  &\qquad \qquad \qquad \qquad \qquad \qquad \qquad \qquad \qquad \qquad  \textrm{since}  \; \tilde{V} \geq 1 \notag \\
  &\leq C_a \tilde{V}^{\frac{n+4}{n+2}} (3nC_b (\norm{R}{\frac{n+2}{2},D}+B) +  2nB^2) \notag \\
  &\qquad \qquad \qquad \qquad \qquad \qquad \qquad \qquad \qquad \qquad  \textrm{note that}  \; B^2 \leq B  \notag \\
  &\leq 3n(C_b+1)C_a \tilde{V}^{\frac{n+4}{n+2}} (\norm{R}{\frac{n+2}{2},D}+B).
  \label{eqn: control5}
\end{align}

Note that $\norm{R_{+}}{\infty,D'} \leq \norm{\hat{R}}{\infty,D'}$.
Let $C(n,\sigma,r) \triangleq
3n(C_b+1)C_a\tilde{V}^{\frac{n+4}{n+2}} $,
 from inequality (\ref{eqn: control5}), we have
\begin{align*}
  \norm{R_{+}}{\infty,D'} \leq C(n,\sigma,r)
  (\norm{R}{\frac{n+2}{2},D}+B).
\end{align*}
\end{proof}

\section{ Proof of Theorem~\ref{theorem: 1} for $n \geq 3$}
\begin{proof}
Since $\norm{R}{\alpha, M \times [0,T)} < \infty $ implies
$\norm{R}{\frac{n+2}{2}, M \times [0,T)} < \infty$ if $\alpha >
\frac{n+2}{2}$, so we only need to prove Theorem~\ref{theorem: 1}
for $\alpha = \frac{n+2}{2}$. We shall argue by contradiction.

Suppose the flow cannot be extended, then $|Ric|$ is unbounded by
Sesum's result. Since $Ric \geq -A$, we know
\begin{align*}
  \sup_{(x,t) \in M \times [0,T)} R(x,t) = \infty.
\end{align*}
Therefore, there exists a sequence $(x^{(i)},t^{(i)})$ such that
$\lim_{i \to \infty} t^{(i)} =T$, and
\begin{align*}
   R(x^{(i)},t^{(i)}) = \max_{(x,t) \in M \times [0,t^{(i)}]}
   R(x,t).
\end{align*}
Consequently, $\lim_{i \to \infty} R(x^{(i)},t^{(i)}) = \infty$.
Define
\begin{align*}
Q^{(i)} \triangleq R(x^{(i)},t^{(i)}), \quad P^{(i)}
  \triangleq B_{g(t^{(i)})}(x^{(i)}, (Q^{(i)})^{-\frac12}) \times [t^{(i)}- (Q^{(i)})^{-1}, t^{(i)}],
\end{align*}
then for any $(x,t) \in P^{(i)}$,   $R(x,t) \leq  Q^{(i)}$.

Now, let $g^{(i)}(t) \triangleq Q^{(i)}
g((Q^{(i)})^{-1}(t-1)+t^{(i)})$. We have a sequence of Ricci flow
solutions: $\{ (M^n,g^{(i)}(t)), \; 0 \leq t \leq 1 \}$. Moreover,
\begin{align}
 & R^{(i)}(x,t) \leq 1, \quad \forall \; (x,t)
\in B_{g^{(i)}(1)}(x^{(i)}, 1) \times [0,1]; \notag\\
 & Ric^{(i)}(x,t) \geq -\frac{A}{Q^{(i)}},  \quad  \forall \; (x,t) \in M \times [0,1].
\label{eqn: riccircondition}
\end{align}
Since $Ric^{(i)}+ \frac{A}{Q^{(i)}} \geq 0$, so
\begin{align*}
Ric^{(i)}+ \frac{A}{Q^{(i)}} \leq tr(Ric^{(i)}+ \frac{A}{Q^{(i)}}) =
R^{(i)}+\frac{nA}{Q^{(i)}}.
\end{align*}
Consequently, $Ric^{(i)} \leq R^{(i)}+\frac{(n-1)A}{Q^{(i)}}$. Note
that $\lim_{i \to \infty} \frac{A}{Q^{(i)}} =0, \; n \geq 3$, by
inequalities (\ref{eqn: riccircondition}), we get
\begin{align}
 & Ric^{(i)}(x,t) \leq n-1, \quad \forall \; (x,t)
\in B_{g^{(i)}(1)}(x^{(i)}, 1) \times [0,1]; \notag\\
 & Ric^{(i)}(x,t) \geq -\frac{A}{Q^{(i)}},  \quad  \forall \; (x,t) \in M \times [0,1].
\label{eqn: riccicondition}
\end{align}

Since for any $x \in B_{g(t^{(i)})}(x^{(i)}, (Q^{(i)})^{-\frac12})$,
$-nA \leq R(x,t) \leq  Q^{(i)}$, for large $i$, we have
 $ |R(x,t)| \leq  Q^{(i)}$. By Theorem~\ref{theorem: noncollapsing2},
there exists a $\kappa$ such that
\begin{align}
    \Vol_{g^{(i)}(1)}(B_{g^{(i)}(1)}(x^{(i)},1))
 = \frac{\Vol_{g(t^{(i)})} (B_{g(t^{(i)}}(x^{(i)}, (Q^{(i)})^{-\frac12}))}
    {(Q^{(i)})^{-\frac{n}{2}}}
 \geq \kappa.
\label{eqn: noncollapsing1}
\end{align}
From inequalities (\ref{eqn: riccicondition}) and (\ref{eqn:
noncollapsing1}),
 we are able to use Theorem~\ref{thm: sobolev}. Therefore, we get a constant $r(\kappa, n)$
  such that for large $i$, on the geodesic ball $B_{g^{(i)}(1)}(p,r)$ ,
   there is a uniform Sobolev constant $\sigma(n,r)$ for every time slice $t \in [0,1]$.

Then we collect conditions to use Theorem~\ref{thm: scalarestimate}.
Define
\begin{align*}
 &\Omega^{(i)} \triangleq B_{g^{(i)}(1)}(p,r),
 &{\Omega^{(i)}}' \triangleq B_{g^{(i)}(1)}(p,\frac{r}{2}),\\
 &D^{(i)} \triangleq \Omega^{(i)} \times [0,1],
 &{D^{(i)}}' \triangleq {\Omega^{(i)}}' \times [\frac12,1].
\end{align*}

 Since $\int_0^T \int_M |R|^{\frac{n+2}{2}} d\mu dt$ is a scale invariant,
\begin{align*}
  \lim_{i \to \infty} \; \norm{R^{(i)}}{\frac{n+2}{2},D^{(i)}}+ \frac{A}{Q^{(i)}}
&= \lim_{i \to \infty} \; \norm{R^{(i)}}{\frac{n+2}{2},D^{(i)}}\\
&= \lim_{i \to \infty} \; \int_{t^{(i)}-(Q^{(i)})^{-1}}^{t^{(i)}}
   \int_{B_{g(t^{(i)})}(p,r(Q^{(i)})^{-\frac12})} |R|^{\frac{n+2}{2}} d\mu dt\\
&\leq \lim_{i \to \infty} \; \int_{t^{(i)}-(Q^{(i)})^{-1}}^{t^{(i)}}
   \int_M |R|^{\frac{n+2}{2}} d\mu dt\\
&= 0.
\end{align*}
The last step comes from $\int_0^T \int_M |R|^{\frac{n+2}{2}} d\mu
dt < \infty$ and $\lim_{i \to \infty} (Q^{(i)})^{-1}=0$.
Consequently, for large $i$, $\norm{R^{(i)}}{\frac{n+2}{2},D^{(i)}}+
\frac{A}{Q^{(i)}} \leq \delta(n,\sigma,r)$. From Theorem~\ref{thm:
scalarestimate}, we know
\begin{align}
     \norm{R_{+}^{(i)}}{\infty,D'} \leq C(n,\sigma,r)(\norm{R^{(i)}}{\frac{n+2}{2},D^{(i)}}
             +\frac{A}{Q^{(i)}}) \;.
\label{eqn: scalarlimit}
\end{align}
Taking limit on both sides, we get
\begin{align*}
  \lim_{i \to \infty} \norm{R_{+}^{(i)}}{\infty,D'}
  \leq  \lim_{i \to \infty}  C(n,\sigma,r)(\norm{R^{(i)}}{\frac{n+2}{2},D^{(i)}}
             +\frac{A}{Q^{(i)}})=0.
\end{align*}
On the other hand,
\begin{align*}
 \lim_{i \to \infty} \norm{R_{+}^{(i)}}{\infty,D'}
 \geq  \lim_{i \to \infty} R_{+}^{(i)}(x^{(i)},1) = 1.
\end{align*}
Therefore we get a contradiction.
\end{proof}

\begin{remark}
   From the proof, we know the condition $Ric \geq -A$ is used only
   to assure that after blowup, Ricci curvature becomes almost
   nonnegative. However, when $\dim =3$, this can be achieved automatically. Actually,
    by Hamilton-Ivey's pinch[ cf.\cite{Ha5}, Theorem4.1],
\begin{align*}
    R \geq |\nu| (\log |\nu| + \log (1+t) -3).
\end{align*}
Here $\nu(x,t)$ is the smallest eigenvalue of the curvature operator
and we have normalized the initial metric such that $inf_{x \in M}
\nu(x, 0) \geq -1$. This tells us that Ricci curvature must be
nonnegative after blowup. Therefore, we can get the
   following Corollary.
\end{remark}

\begin{corollary}
$\{ (M^3,g(t)), 0 \leq t <T < \infty \}$ is a closed Ricci flow
solution.
 If $\norm{R}{\alpha,\; M \times [0,T)} < \infty, \;
 \alpha \geq \frac{5}{2}$,
 then this flow can be extended over time $T$.
\label{corollary: 1}
\end{corollary}

A natural question is  whether the Ricci lower bound condition
superfluous in higher dimension.  To be conservative, can we
substitute the condition $Ric \geq -A$ by a weaker one?

\vspace{0.5in}

University of Wisconsin Madison,  Department of Mathematics

480 Lincoln Drive, Madison, WI, 53706, U.S.A

\textit{E-mail address:} \textsl{bwang@math.wisc.edu}


\begin{thebibliography}{99}

\bibitem{Be1} M.Berger, Lectures on Geodesics in Riemannian
Geometry, Tata Institute, Bombay, 1965.

\bibitem{Be2} A.Besse, Manifolds All of Whose Geodesics are Closed ,
Ergebnisse der Mathematik, Vol. 93, Springer, Berlin-Heidelberg-New
York, 1978.

\bibitem{CT}  Xiuxiong Chen, Gang Tian, Ricci flow on K\"{a}hler-Einstein
manifolds, Duke Math. J. 131 (2006), no. 1, 17--73.

\bibitem{CLN} Bennett Chow, Peng Lu, Lei Ni, Hamilton's Ricci Flow,
preprint.

\bibitem{Cr} Christopher B.Croke, Some Isoperimetric Inequalities
and Eigenvalue Estimates, Ann. scient.Ec. Norm. Sup., 1980, 419-435.

\bibitem{De} D.DeTurck: Deforming metrics in the direction of their
Ricci tensors, J.Differential Geometry. 18(1983), no.1, 157-162.

\bibitem{Ha1} R.S.Hamilton, Three-manifolds with positive Ricci
curvature, J.Differential Geometry. 17(1982), no.2, 255-306.

\bibitem{Ha2} R.S.Hamilton, Four-manifolds with positive curvature
operator, J.Differential Geometry. 24(1986), no.2, 153-179.

\bibitem{Ha3} R.S.Hamilton, The formation of singularities in the
Ricci flow, Surveys in Differential Geometry, vol.2, International
Press, Cambridge, MA(1995) 7-136.

\bibitem{Ha4} R.S.Hamilton, A compactness property for solutions of
the Ricci flow, Amer.J.Math.117(1995) 545-572.

\bibitem{Ha5} R.S.Hamilton, Non-singular solutions of the Ricci flow
on three manifolds, Commun. Anal. Geom. 7(1999), 695-729.

\bibitem{KL} Bruce Kleiner, John Lott, Notes on Perelman's papers,
arXiv: math.DG/0605667.

\bibitem{MT} John W. Morgan, Gang Tian, Ricci Flow and the
 Poincar\`e Conjecture, arXiv: math.DG/0607607.

\bibitem{Pe1} Grisha Perelman, The entropy formula for the Ricci
flow and its geometric applications, arXiv: math.DG/0211159.

\bibitem{Pe2} Grisha Perelman, Ricci flow with surgery on
three-manifolds, arXiv: math.DG/0303109.

\bibitem{Pe3} Grisha Perelman,Finite extinction time for the
 solutions to the Ricci flow on certain three-manifolds,
 arXiv: math.DG/0307245.

\bibitem{Sa} L.A. Santalo, Integral Geometry and Geometric
Probability, Encyclopedia of Mathematics and Its Applications,
Addison-Wesley, London-Amsterdam-Don Mills,
Ontario-Sydney-Tokyo,1976.

\bibitem{Se} Natasa Sesum, Curvature tensor under the Ricci flow,
arXiv: math.DG/0311397.

\bibitem{ST} Natasa Sesum, Gang Tian,Bounding scalar curvature and
diameter along the Kaehler-Ricci flow (after Perelman) and some
applications,
http://www.math.lsa.umich.edu/~lott/ricciflow/perelman.html

\bibitem{SY} R. Schoen, S.-T. Yau, Lectures on differential
geometry, Cambridge, MA , International Press, c1994.


\end{thebibliography}
\end{document}